\documentclass[a4paper,12pt]{amsart}

\usepackage{mathtools}
\usepackage{amsthm}
\newtheorem{lemma}{Lemma}[section]
\newtheorem{theorem}[lemma]{Theorem}
\newtheorem{corollary}[lemma]{Corollary}
\newtheorem{proposition}[lemma]{Proposition}
\theoremstyle{definition}
\newtheorem{definition}[lemma]{Definition}
\newtheorem{remark}[lemma]{Remark}

\usepackage{amsmath, amssymb, amsfonts, stmaryrd}
\date{}
\usepackage{enumerate}
\usepackage{hyperref}
\usepackage{a4wide}

\title{$L^2$-instability of the Taub-Bolt metric under the Ricci flow}
\author{John Hughes}
\address{Mathematical Institute, University of Oxford, Oxford OX2 6GG, United Kingdom}
\email{john.hughes@maths.ox.ac.uk}

\begin{document}

		\begin{abstract}
		In this paper we prove that there exists a compact perturbation of the Ricci flat Taub-Bolt metric that evolves under the Ricci flow into a finite time singularity modelled on the shrinking solition FIK. Moreover, this perturbation can be made arbitrarily $L^2$-small with respect to the Taub-Bolt metric. The method of proof closely follows the strategy adopted in a paper of Stolarski, where, via a Wa{\.z}ewski box argument, Ricci flows on compact manifolds which encounter finite singularities modelled on a given asymptotically conical shrinking soliton are constructed.
	\end{abstract}
			\maketitle

	\section{Introduction}
	If $(M,g_{0})$ is a Riemannian manifold, the Ricci flow is an equation for the evolution of~$g_{0}$:
	\begin{equation}\label{RFequation}
		\partial_{t}g(t)=-2\mathrm{Ric}(g(t)), \text{ } g(0)=g_{0}.
	\end{equation}
An important class of examples of families of metrics $g(t)$ satisfying (\ref{RFequation}) are Ricci solitons. A Ricci soliton $(M,g,X)$ is a triple consisting of a Riemannian manifold $(M,g)$ and a vector field $X$ on $M$ such that 
\begin{equation}\label{solitoneq}
	\textrm{Ric}+ \mathcal{L}_{X}g=\lambda g,
\end{equation}
for some $\lambda \in \mathbb{R}$. Let $\phi_{t}:M\rightarrow M$ be the family of diffeomorphisms satisfying 
\begin{equation}\label{phi_t}
	\partial_{t} \phi_{t}=\frac{1}{1-2\lambda t} X \circ \phi_{t}, \qquad \phi_{0}=\mathrm{Id}_{M}.
\end{equation}
 Then a solution to the Ricci flow equation (\ref{RFequation}) is given by $g(t)=(1-2\lambda t)\phi_{t}^* g$. If $X=\nabla f$ for some function $f$ then $(M,g,X)=(M,g,f)$ is called a gradient Ricci soliton. 
An example of a Ricci soliton central to this paper is $(\mathbb{CP}^2\backslash\{\text{pt}\}, g_{\mathrm{FIK}}, f_{\mathrm{FIK}})$, a K\"{a}hler shrinking ($\lambda>0$) gradient Ricci soliton called FIK \cite{FIK}.

	The fixed points of equation (\ref{RFequation}) are precisely the Ricci flat metrics. It is therefore natural to ask about the stability of such metrics under the Ricci flow. The purpose of this paper is to prove the following dynamical instability result for the Ricci flat Taub-Bolt metric $(\mathbb{CP}^2\backslash\{\text{pt}\}, g_{\mathrm{Bolt}})$~\cite{Bolt}: 
	\begin{theorem}\label{finitesing}
		Let $\epsilon>0$. There exists a Ricci flow starting at a compact perturbation $G(t_0)$ of $g_{\mathrm{Bolt}}$
		such that $\lVert G(t_0)-g_{\mathrm{Bolt}} \rVert_{L^2_{g_{\mathrm{Bolt}}}}<\epsilon$, and 
		 which encounters a local finite time singularity modelled on the shrinking soliton FIK. 
	\end{theorem}

The dynamical instability of $g_{\mathrm{Bolt}}$ is not unexpected. Indeed, in \cite{Cla1}, $g_{\mathrm{Bolt}}$ is shown to be linearly unstable, meaning the Lichnerowicz Laplacian $\Delta_{g_{\mathrm{Bolt}}}^L$ has a positive eigenvalue with associated smooth $L^2_{g_{\mathrm{Bolt}}}$ eigentensor $h$, which we call an unstable direction. Furthermore, numerical simulations presented in \cite{Cla1}  suggest that if the Ricci flow is started at $g_{\mathrm{Bolt}}+h$, then a finite time singularity forms, caused precisely by the bolt shrinking to zero size in finite time. Similar numerical findings for the Ricci flow starting at more general compact perturbations of $g_{\mathrm{Bolt}}$ are presented in \cite{Oxford1}. It is widely believed that the singularity is modelled on the asymptotically conical shrinking soliton FIK, which has the same underlying manifold as $g_{\mathrm{Bolt}}$. 

Actually proving the numerical findings of \cite{Cla1} and \cite{Oxford1} analytically is challenging. Even proving that the Ricci flow starting at some compact or $L^2_{g_{\mathrm{Bolt}}}$-small perturbation of $g_{\mathrm{Bolt}}$ leads to a finite time singularity is not straightforward. Indeed, in many other scenarios, see \cite{App1}, \cite{Dan2} or \cite{Fra2} for example, showing the existence of a finite time singularity follows easily from a simple use of the maximum principle, with the hard work only  occurring in a more detailed analysis of the singularity formation. Being Ricci flat, $g_{\mathrm{Bolt}}$ is fixed under the Ricci flow. Consequently a global maximum principle cannot be applied to a Ricci flow starting at a compact or $L^2_{g_{\mathrm{Bolt}}}$-small perturbation of $g_{\mathrm{Bolt}}$. This means proving Theorem \ref{finitesing} will require a more subtle argument. As a result, this paper will proceed via a non-constructive approach known as a Wa{\.z}ewski box argument \cite{Wazewski} to prove Theorem \ref{finitesing}. We will closely follow the application of the Wa{\.z}ewski box argument by Stolarksi \cite{Stol1}, where it is shown that given an asymptotically conical shrinking soliton, there is a Ricci flow on a compact manifold which forms a local finite time singularity modelled on it. The family of metrics used in the Wa{\.z}ewski box argument \cite{Stol1} is a  family of perturbations of the given soliton glued onto the two ends of a finite cylinder. Inspired by this, we will use a family of perturbations of $g_{\mathrm{FIK}}$ glued into $g_{\mathrm{Bolt}}$. 

We emphasize that despite $g_{\mathrm{FIK}}$ and  $g_{\mathrm{Bolt}}$ both possessing the same cohomogeneity one $U(2)$-symmetry on $\mathbb{CP}^2\backslash\{\text{pt}\}$, the Ricci flow $G(t)$ provided by Theorem \ref{finitesing} will not necessarily be a flow through cohomogeneity one $U(2)$-symmetric metrics. 

A outline of the paper is as follows. Section \ref{TBandFIK} will introduce the Ricci flat Taub-Bolt metric $g_{\mathrm{Bolt}}$ and the shrinking soltion FIK. Section \ref{setup} will set up the framework of the Wa{\.z}ewski box argument used to prove Theorem \ref{finitesing}. Section \ref{final} will prove Theorem \ref{finitesing}, which will be a proof of distinct parts. The first part is very closely follows Stolarkski Wa{\.z}ewski box argument to show the existence of a Ricci flow with initial condition that is a compact perturbation of $g_{\mathrm{Bolt}}$ which encounters a finite time singularity modelled on FIK. The second part is showing that we can actually choose this initial condition to be as $L^2_{\mathrm{Bolt}}$-close to $g_{\mathrm{Bolt}}$ as possible. Appendix \ref{construct} will describe precisely the family of initial metrics used in the application of the Wa{\.z}ewski box argument. Appendix \ref{pseudoapp} will prove a pseudolocality result needed to prove that the finite time singularity of Theorem \ref{finitesing} is local. \\

$\mathbf{Acknowledgements.}$ I would like to thank my
supervisor Jason Lotay for his support and advice. Research supported by a scholarship from EPSRC (grant number EP/W524311/1).

\section{The Taub-Bolt and FIK metrics}\label{TBandFIK}
In this section we will introduce the Taub-Bolt and FIK family of metrics.

Let $\sigma_{1}, \sigma_{2}, \sigma_{3}$ be a Milnor frame on $S^3$. Set $M=\mathbb{CP}^2\backslash\{\text{pt}\}$. Then $M$ is diffeomorphic to the complex line bundle $\mathcal{O}(-1)$ over $\mathbb{CP}^1$. Furthermore,  cohomogeneity one $U(2)$-symmetric metrics on $M$ can be written in the form 
\begin{equation}\label{cohomform}
	g=ds^2+ b^2(s)\left(\sigma_{1}^2+\sigma_{2}^2\right)+c^2(s)\sigma_{3}^2, \qquad s>0,
\end{equation}
where the metric closes up at $s=0$ to produce a well defined metric on $M$ provided $b(0)>0, c(0)=0$, $\partial_s b(0)=~0$ and $\partial_s c(0)=1$. The set $\{s=0\}\subset M$ is a copy of $\mathbb{CP}^1=S^2$ and is called the bolt of $g$.

\subsection{Taub-Bolt}

Strictly speaking the Taub-Bolt metric is a family of metrics $\alpha g_{\mathrm{Bolt,n}}$ parametrised by $\alpha, n\in \mathbb{R}_{>0}$. The metric $g_{\mathrm{Bolt,n}}$ is a Ricci flat metric on $M$ with a $U(2)$-symmetry acting by cohomogeneity one, and so can be written in the form (\ref{cohomform}) with $b\eqqcolon b_{\text{Bolt,n}}, c\eqqcolon c_{\text{Bolt,n}}$. 
Reparametrising the radial direction $\partial_s$, the metric $g_{\mathrm{Bolt,n}}$ can be written explicitly as 
\begin{equation}\label{TB}
	g_{\mathrm{Bolt},n}=\frac{(r+n)(r+3n)}{r\left(r+\frac{3n}{2}\right)}dr^{2}+ 4(r+n)(r+3n)\left(\sigma_{1}^{2}+\sigma_{2}^{2}\right)+ 16n^{2}\frac{r\left(r+\frac{3n}{2}\right)}{(r+n)(r+3n)}\sigma_{3}^2.
\end{equation}
The relation between $r>0$ and $s$ is given by $s= \int_{0}^{r} \sqrt{\frac{(\bar{r}+n)(\bar{r}+3n)}{\bar{r}\left(\bar{r}+\frac{3n}{2}\right)}} d\bar{r}$.
 The coefficients satisfy $\partial_s b_{\mathrm{Bolt,n}}(s) \rightarrow 2$ and $c_{\mathrm{Bolt,n}}(s) \rightarrow 4n$ as $s\rightarrow \infty$. Furthermore, $\lvert \text{Rm}_{g_{\mathrm{Bolt,n}}} \rvert_{g_{\mathrm{Bolt,n}}}(s)=O(s^{-3})$ as $s\rightarrow \infty$. From (\ref{TB}) we can see that the asymptotic geometry of $g_{\mathrm{Bolt,n}}$ is a cone over a Berger sphere collapsed in the $\sigma_{3}$ direction, and the volume growth is cubic in $s$.

 Up to rescaling and pullback by diffeomorphisms, the Taub-Bolt metric family has only one member. Indeed, if  we rescale the fibres of this $M=\mathcal{O}(-1)$ according to the diffeomorphism $\phi(r)=\frac{n}{m}r$ of $\mathbb{CP}^2\backslash\{\text{pt}\}$ then $\phi^{*}g_{\mathrm{Bolt},n}=\frac{n^2}{m^2}g_{\mathrm{Bolt},m}$. Since the Ricci flow is invariant under rescaling and pullback by diffeomorphisms, we fix $n$ and  denote $g_{\mathrm{Bolt,n}}$ by $g_{\mathrm{Bolt}}$.

\subsection{FIK}

We introduce FIK via the following theorem.
\begin{theorem}[\cite{FIK}]\label{FIKthm}
	Let $C_{\mathrm{FIK}}\coloneqq (C(S^3), g_{C_{\mathrm{FIK}}})$ be the cone over the sphere $(S^3, g_{S^3_\mathrm{FIK}} \coloneqq \frac{1}{\sqrt{2}}(\sigma_{1}^2+\sigma_{2}^2)+\frac{1}{2}\sigma_{3}^2)$. Then there exists a $U(2)$-symmetric shrinking gradient K\"{a}hler Ricci soliton $(M, g_{\mathrm{FIK}}, f_{\mathrm{FIK}})$, with $\lambda=\frac{1}{2}$ in equation (\ref{solitoneq}), on $M=\mathbb{CP}^2\backslash\{\mathrm{pt}\}$ that is asymptotically conical with asymptotic cone $C_{\mathrm{FIK}}$. Let $\phi_{t}$ be the family of diffeomorphisms satisfying (\ref{phi_t}) with $X=\nabla f_{\mathrm{FIK}}$, and let $g_{\mathrm{FIK}}(t)=(1-t)\phi^*_t g_{\mathrm{FIK}}(0)$ denote the Ricci flow with initial condition $g_{\mathrm{FIK}}$. There exists a diffeomorphism $\Psi : C_{\mathrm{FIK}}  \to (\mathbb{CP}^2\backslash \{\mathrm{pt}\})\setminus \{s=0\}$ which maps $\{r=\mathrm{constant}\}\subset C_{\mathrm{FIK}}$ to $\{f_{\mathrm{FIK}}=\mathrm{constant}\} \subset(\mathbb{CP}^2\backslash \{\mathrm{pt}\})\setminus \{s=0\}$, and preserves the order. Also, the  smooth family of diffeomorphisms $\Psi_t =  \phi_t \circ \Psi$ yields 
	$$ ( 1 - t) \Psi_t^*  g_{\mathrm{FIK}} \xrightarrow[t \nearrow 1]{ C^\infty_{loc} ( C(S^3), g_{C_{\mathrm{FIK}}} ) }		g_{C_{\mathrm{FIK}}}.$$
	Finally, when $g_{\mathrm{FIK}}$ is written in the form 
	\begin{equation*}
		g_{\mathrm{FIK}}=ds^2+ b_{\text{FIK}}^2(s)\left(\sigma_{1}^2+\sigma_{2}^2\right)+c_{\text{FIK}}^2(s)\sigma_{3}^2, \qquad s>0,
	\end{equation*}
	we have the following:
	\begin{enumerate}[i)]
		\item $\partial_s f_{\mathrm{FIK}}(0)=0$ and $\partial_s {f_{\mathrm{FIK}}}>0$ for $s>0$;
		\item $b_{\mathrm{FIK}}(s) \rightarrow \frac{s}{\sqrt[4]{2}}$, $c_{\mathrm{FIK}}(s) \rightarrow \frac{s}{\sqrt{2}}$ as $s\rightarrow \infty$
		\end{enumerate}
\end{theorem}

\begin{remark}\label{infinity} 
	The diffeomorphisms $\phi_t$ are generated by $\frac{\nabla f_{\mathrm{FIK}}}{1-t}$, and $\phi_0=\mathrm{Id}_M$. Thus, $\phi_t$ preserves the level sets of $f_{\mathrm{FIK}}$. 
	Since $g_{\mathrm{FIK}}(t)=(1-t)\phi^*_t g_{\mathrm{FIK}}(0)$ and the curvature only blows up at the bolt of $g_{\mathrm{FIK}}$ \cite{FIK}, the diffeomorphims $\phi_{t}$ have to push all points with $s>0$ to infinity in finite time, while $\{s=0\}\subset M$ is fixed. This implies that for any $C\geq \tilde{C}>f_{\mathrm{FIK}}(\{s=0\})$ there is a $t\in[0,1)$ such that $\phi_{t}(\{f_{\mathrm{FIK}}=\tilde{C}\})= \{f_{\mathrm{FIK}}=C\}$.
\end{remark}
	
	To make the referencing of \cite{Stol1} as easy as possible we will use the same notation, and so for the rest of this paper we will fix the notation $(M, \overline{g}, f)\coloneqq ( \mathbb{CP}^2\backslash\{\text{pt}\}, g_{\mathrm{FIK}}, f_{\mathrm{FIK}})$ with $\phi_{t}$ the assoicated family of diffeomorphisms. Furthermore, the Levi-Civita connection associated with $\overline{g}=g_{\mathrm{FIK}}$ will be denoted by $\overline{\nabla}$. For the rest of this paper, $f_{\mathrm{Bolt}} \coloneqq f(\{s=0\})$ (i.e the value of $f$ at the bolt of $(M, g_{\mathrm{FIK}})$). The soliton potential $f$ is unique up a constant. To use \cite[Lemma 2.33]{Stol1} without alteration, we follow Stolarksi and fix the constant by demanding that $\mathrm{R}+\lvert \overline{\nabla} f \rvert_{\overline{g}}^2=f$ on $M$, where $R$ is the scalar curvature of $\overline{g}$.

	\section{The setup}\label{setup}
	Following \cite{Stol1} closely, in this section we will set up the Wa{\.z}ewski box argument.
	
	Fix a smooth $L_{f}^2$-orthonormal basis $\{h_{j}\}_{j=1}^{\infty}$ of eigenmodes of the weighted Lichnerowicz Laplacian with respect to $g_{\mathrm{FIK}}$, $\Delta_{f}+2\overline{Rm}$, with corresponding eigenvalues $\{\lambda_{j}\}_{j=1}^{\infty}$. Choose $\lambda_{*}\in (-\infty,0)\setminus \{\lambda_{j}\}_{j=1}^{\infty}$ and $K\in \mathbb{N}$ such that $\lambda_{K}>\lambda_{*}>\lambda_{K+1}$ \cite[\S 2]{Stol1}.
	
	The rough idea of the argument used to prove Theorem \ref{finitesing} is the following.
	We will construct a metric $G_{\mathbf{p}}(t_0, \Gamma_{0})$ which interpolates between the metric
	\begin{equation*}
		g= (1-t_{0})\phi_{t_{0}}^{*}\biggl[\overline{g}+\eta_{\gamma_{0}}\sum_{j=1}^{K} p_{j}h_{j} \biggr],
	\end{equation*} 
	
	 the cone $g_{C_{\mathrm{FIK}}}$, and $g_{\mathrm{Bolt}}$, 
	where $\mathbf{p}=(p_{1}, ..., p_{K}) \in \mathbb{R}^{K}$,  $t_{0}<1$, $\eta_{\gamma_{0}}$ is a smooth bump function, and $\Gamma_{0}$ is a parameter which controls the location on which the metrics $G_{\mathbf{p}}(t_0, \Gamma_{0})$ and $g_{\mathrm{Bolt}}$ differ (smaller $\Gamma_{0}$ implies they differ over a smaller region). Let $G_{\mathbf{p}}(t, t_0, \Gamma_{0})$ denote the maximal solution to the Ricci flow on $M$ with initial condition $G_{\mathbf{p}}(t_0, \Gamma_{0})$. We will show that for each $\Gamma_{0}\leq C$, where $C$ is some constant whose value will be determined in the course of the proof of Theorem \ref{finitesing}, there is some $\mathbf{p}$ and $t_{0}$ such that the family of metrics $G_{\mathbf{p}}(t, t_0, \Gamma_{0})$ gives us a Ricci flow described in Theorem \ref{finitesing}. By choosing $\Gamma_{0}$ to be arbitrarily small, we will be to make the region on which $G_{\mathbf{p}}(t_0, \Gamma_{0})$ and $g_{\mathrm{Bolt}}$ differ be arbitrarily small. This will allows us to show that $\lVert G_{\mathbf{p}}(t_0, \Gamma_{0})-g_{\mathrm{Bolt}} \rVert_{L^2_{g_{\mathrm{Bolt}}}}$ can be made arbitrarily small.
	
	\subsection{$G_{\mathbf{p}}(t_{0})$}\label{first}

	In Appendix \ref{construct}, a family of metrics on $M$ with the following properties will be constructed.
	\begin{definition} \label{G_0} There exists a $\tilde{\Gamma}_0>0$ such that for every $0<\Gamma_{0}\leq \tilde{\Gamma}_0$  and every $0 < 1- t_0 \ll 1$ sufficiently small depending on $\Gamma_{0}$, there exists a
		Riemannian metric $G_{\mathbf 0}( t_0) = G_{\mathbf 0}( t_0, \Gamma_{0})$ 
		on $M$ such that:
		\begin{enumerate}[i)]
			\item 
			(FIK near the bolt of $(M,\overline{g})$) 
			$$G_{\mathbf 0} (t_0) = ( 1 -t_0) \phi_{t_0}^* g
			\qquad \text{on } \{ x \in M : f(x) \le f_{\mathrm{Bolt}}+\Gamma_{0} \},$$

			\item 
			(Independence of $t_0$ far from $\{f=f_{\mathrm{Bolt}}\}$)
			Throughout the subset $\{ x \in M :  f(x) \ge f_{\mathrm{Bolt}}+ 16\Gamma_{0} \}$,
			$$G_{\mathbf 0 } ( t_0) = G_{\mathbf 0 } ( t_0 ' )	\qquad 
			\text{for all } 0 \le t_0, t_0 ' < 1,$$

			\item 
			(Convergence to cone in an intermediate region)
			The assignment $t_0 \mapsto G_{\mathbf 0} (t_0)$ is smooth in $t_0$. Furthermore, restricting the embedding from Theorem \ref{FIKthm} (and denoting its inverse also by $\Psi$), 
			 
			$$\Psi : \{ x \in M : f_{\mathrm{Bolt}}< f(x) < f_{\mathrm{Bolt}}+32\Gamma_0 \} \rightarrow C_{FIK},$$
			we have that $G_{\mathbf{0}}(t_0)$ smoothly converges to $\Psi^* g_{C_{\mathrm{FIK}}}$ on this region as $t_0 \nearrow 1$.

			\item 
			(Curvature estimates far from $\{f=f_{\mathrm{Bolt}}\}$)
			For all $m \in \mathbb{N}$, there exists $C=C(m,\Gamma_{0})$
			such that,
			if $0 < 1 - t_0 \ll 1$ is sufficiently small,
			then 
			$$| \overline{\nabla}^m \mathrm{Rm}[ G_{\mathbf 0 }(t_0) ] | \le  C$$
			on the set  $\{x \in M : f(x) \ge f_{\mathrm{Bolt}}+\Gamma_0 /2 \} $.
			
			\item (Glueing into $g_{\mathrm{Bolt}}$) For some $\tilde{C}=\tilde{C}(\Gamma_{0})\geq 32\Gamma_{0}>0$ and $0 < 1 - t_0 \ll 1$ sufficiently small depending on $\Gamma_{0}$, we have
			$G_{\mathbf{0}}(t_0)=g_{\mathrm{Bolt}}$ on $\{f>f_{\mathrm{Bolt}}+\tilde{C}\}$. Furthermore, $\tilde{C}(\Gamma_{0}) \rightarrow 0$ as $\Gamma_{0} \rightarrow 0$.
			
			\item There is a $B>0$ independent of $\Gamma_{0}\leq \tilde{\Gamma}_0$ such that $G_{\mathbf{0}}(t_0) \leq B g_{\mathrm{Bolt}}$ on $M$.
		\end{enumerate}
	\end{definition}
Since we are following the set up of \cite{Stol1}, we will, apart from the exception on one parameter, use the same notation as Stolarski. That one exception is that the value of the parameter $\Gamma_{0}$ will change from $\Gamma_{0}>f_{\mathrm{min}}$ as is the case in \cite{Stol1} to $\Gamma_{0}>0$. We do this as, unlike Stolarski, we are interested, for reasons that will be explained in the discussion of property v) below, in the case where the region on which $G_{\mathbf{0}}(t_0, \Gamma_{0})$ differs from $g_{\mathrm{Bolt}}$ to be small, which in our notation corresponds to small $\Gamma_{0}$ (we prefer to think of $\Gamma_{0}$ as small as opposed to $\Gamma_{0}$ near to $f_{\mathrm{min}}=f_{\mathrm{Bolt}}<\Gamma_{0}$). With this notation change in mind we discuss the properties i)-vi).
Stolarski only uses the properties of the metric $G_{\mathbf{0}}$ listed in \cite[Definition 4.2]{Stol1} to prove \cite[Theorem 1.1]{Stol1}. 
Similarly, the proof of Theorem \ref{finitesing} will not rely on any other properties of $G_{\mathbf{0}}$ other than those listed in Definition \ref{G_0}.  
Some of the properties listed in Definition \ref{G_0} are akin to the properties listed in \cite[Definition 4.2]{Stol1}. 
 Properties i), ii) and iii) are exactly the same. Property iv) only differs in that the constant $C$ in \cite{Stol1} is not dependent on $\Gamma_{0}$, but the right hand side of the inequality is instead $\frac{C(m)}{\Gamma_{0}^{1+\frac{m}{2}}}$. This difference does not have any meaningful impact on our use of the results proven in \cite{Stol1} as, after an initial set up, the author fixes $\Gamma_{0}$. Property v) is written as we are glueing metrics into $g_{\mathrm{Bolt}}$ and want the region on which we do this to be as small as possible, which is the main difference between the work in \cite{Stol1} and the work in this paper. Making the region on the which the glueing procedure happens arbitrarily small will allow us to, by choosing $\Gamma_{0}$ to arbitrarily small, make the initial metric of the Ricci flow described in Theorem \ref{finitesing} arbitrarily $L^2_{g_{\mathrm{Bolt}}}$ close to $g_{\mathrm{Bolt}}$. Stolarski works with an arbitrary asymptotically conical shrinking soliton and so does not a priori know the shrinker converges to its asymptotic cone globally. However, as mentioned above, this is the case for FIK. Thus, we can take the conical metric and the analysis of the intermediate region in part iii) of Definition \ref{G_0} conducted in \cite{Stol1} as close to the set $\{f=f_{\mathrm{Bolt}}\}$ as possible. As a result, we can choose to have the region on which $G_{\mathbf{0}}(t_0)$ differs from $g_{\mathrm{Bolt}}$ be arbitrarily small (i.e. allow $\Gamma_{0}$ to be arbitrarily small). In order to do this, unlike in \cite{Stol1}, the parameter $t_0$ needs to depend on the value of $\Gamma_{0}$, so that by taking $t_0$ close enough to 1, we can ensure property iii) in Definition \ref{G_0} is satisfied. The purpose of property vi) is to control the geometry of $G_{\mathbf{0}}(t_0)$ and will also be useful to show that the initial metric of the Ricci flow described in Theorem \ref{finitesing} arbitrarily $L^2_{g_{\mathrm{Bolt}}}$ close to $g_{\mathrm{Bolt}}$. 
 
 As a consequence of this main difference, we will need to adapt, in minor ways, many of the results in \cite{Stol1}. Most, but not all, of the statements and their proofs of these adapted results will not be written in this paper, but will be used. A more precise description of the how to adapt the statements and proofs will be presented in Subsection \ref{comments}, after we have introduced some more notation.

	Using $G_{\mathbf{0}}(t_0)$ we define the family of initial metrics $G_{\mathbf{p}}(t_0)$. This paper will show that for $0<1-t_0 \ll 1$ sufficiently small and some $\mathbf{p}\in \mathbb{R}^K$, Ricci flow starting at $G_{\mathbf{p}}(t_0)$ gives us a Ricci flow as in Theorem \ref{finitesing}. 
		
	Following Stolarski, let $0<\overline{p}<1$. For all $\mathbf p = ( p_1, \dots , p_{K} ) \in \mathbb{R}^K$ with $| \mathbf p | \le \overline{p} (1 - t_0)^{| \lambda_*|}$, define a symmetric 2-tensor $G_{\mathbf p }(t_0) = G_{\mathbf p } ( \Gamma_0 , \gamma_0, t_0)$ on $M$ by
	\begin{equation} \label{G_p}
		G_{\mathbf p} (t_0 ) \coloneqq G_{\mathbf 0 }(t_0 ) 	
		+ ( 1 - t_0) \phi_{t_0}^* \left( \eta_{\gamma_0}     \sum_{j = 1}^{K } p_j  h_{j}   \right),
	\end{equation}
where, for $\gamma_{0}>0$, $\eta_{\gamma_{0}}$ is the bump function constructed in \cite[\S 4.2]{Stol1} which satisfies $\eta_{\gamma_0} (x) = 1$ for all $x \in M$ such that $f(x) \le \frac{\gamma_0}{2 ( 1 - t_0)}$ and $\text{supp} (\eta_{\gamma_0}) \subset \left \{ x \in M : f(x) < \frac{\gamma_0}{1 - t_0} \right \}$.

	The next couple of results will allow us to conclude that, provided $0 < 1-t_0 \ll 1$ is sufficiently small and $0 < \gamma_0=\gamma_{0}(t_0) \le 1$ is sufficiently small, $G_{\mathbf{p}}(t_0)$ is a well-defined metric on $M$.
	
	As an illustration of how the proofs of many of the results will need to be adapted to allow for small $\Gamma_{0}$, we provide the full details of the proof of the following lemma, which follows the proof of \cite[Lemma 4.5]{Stol1}
	\begin{lemma}\label{lemmaforwelldefined}
		Let $\Gamma>0$. If
		$0 < 1-t_0 \ll 1$ is sufficiently small and $0 < \gamma_0=\gamma_{0}(t_0) \le 1$ is sufficiently small,
		then
		$$\mathrm{supp} (\phi_{t_0}^* \eta_{\gamma_0}) \subset \{x \in M :  f(x) < f_{\mathrm{Bolt}}+\Gamma \}.$$
	\end{lemma}	
\begin{proof}
	Rewriting the conclusion of \cite[Lemma 2.33]{Stol1} using the variable $\tau(t)= -\text{ln}(1-t)$,  it follows that there exists $C>0$ such that 
	$$\max\left\{0,\phi^*_\tau f-\frac{C}{\phi^*_\tau f}\right\}\leq \partial_\tau(\phi^*_\tau f)\leq \phi_{\tau}^*f,$$
	for all $\tau \in \mathbb{R}$. By Remark \ref{infinity}, there exists $\tilde{\tau}_0>0$ such that if $f(x)=f_{\mathrm{Bolt}}+\Gamma$ then $(\phi_{\tilde{\tau}_0}^*f)(x)>\sqrt{C}$. Integrating from $0$ to $\tau$  gives 
	$$f(x)\leq \phi^*_{\tau}f(x), \qquad \forall (x,\tau)\in M\times [0,\infty),$$
	$$\sqrt{C+(f(\phi_{\tilde{\tau}_0}(x))^2-C)e^{2(\tau-\tilde{\tau}_0)}}\leq \phi_{\tau}^*f(x)\qquad \forall (x,\tau)\in \{f>\sqrt{C}\} \times \{(\tilde{\tau}_0, \infty)\}.$$
	Let $\tau_{0}=-\text{ln}(1-t_0)\geq \tilde{\tau}_0$. Notice that 
	$$\text{supp} (\phi_{\tau_0}^*\eta_{\gamma_0})= \{[\phi_{\tau_0}^*f](x)\leq \gamma_{0}e^{\tau_{0}}\}.$$	Let $x\in M$ be such that $\phi_{t_0}^*f(x)\leq \gamma_{0}e^{\tau_{0}}$. If $f(\phi_{\tilde{\tau}_0}(x))<\sqrt{C}$ then $f(x)<\Gamma$. If $f(\phi_{\tilde{\tau}_0}(x))\geq \sqrt{C}$ then 
	$$\sqrt{C+(f(\phi_{\tilde{\tau}_0}(x))^2-C)e^{2(\tau_0-\tilde{\tau_{0}})}}\leq \phi^*_{\tau_{0}}f(x)\leq \gamma_{0}e^{\tau_{0}}.$$
	Therefore,
	$$f(\phi_{\tilde{\tau}_0}(x))\leq \sqrt{C+e^{2\tilde{\tau_{0}}}\gamma_{0}^2-Ce^{2(\tilde{\tau}_0-\tau_{0})}}\leq \sqrt{C+1},$$
	whenever $\gamma_{0}<e^{-2\tilde{\tau_{0}}}$. By the choose of $\tilde{\tau}_0$, it follows that  $f(x)<\Gamma$ must be the case. 	
\end{proof}

	\begin{proposition} \label{G_p-G-0}
		Let $0<\Gamma_{0}\leq \tilde{\Gamma}_0$. If
		$0 \leq 1-t_0 \ll 1$ is sufficiently small and $0 < \gamma_0=\gamma_{0}(t_0, \lambda_{*}) \le 1$ is sufficiently small,	
		then there exists $C=C(\lambda_{*})>0$ such that
		$$| G_{\mathbf p} (t_0) - G_{\mathbf 0} (t_0) |_{G_{\mathbf 0} (t_0)} \leq C\overline{p}  \gamma_0^{|\lambda_*|}$$
		for all $| \mathbf p | \le \overline{p} ( 1 - t_0)^{|\lambda_*|} $.
	\end{proposition}
	\begin{proof}
		By Lemma \ref{lemmaforwelldefined}, we can assume that 
		$\mathrm{supp} (\phi_{t_0}^* \eta_{\gamma_0}) \subset \{x \in M :  f(x) < \Gamma_0 \}$. Following the proof of \cite[Proposition 4.5]{Stol1} exactly, the desired result follows.
	\end{proof}
	
	\begin{corollary}\label{welldefined}
		Let $0<\Gamma_{0}\leq \tilde{\Gamma}_0$. If
		$0 < 1-t_0 \ll 1$ is sufficiently small and $0 < \gamma_0=\gamma_{0}(t_0) \le 1$ is sufficiently small,
		then $G_{\mathbf{p}}(t_0)$ is a well-defined metric on $M$.
	\end{corollary}

	\subsection{The flow $G_{\mathbf{p}}(t, t_{0})$}\label{assumption}
	
	For the rest of Section \ref{setup}, $\Gamma_{0}>0$ will some fixed constant. Assume that $0 < 1-t_0  \ll 1$ is sufficiently small depending on $\Gamma_{0}$; 
	$0 < \gamma_0=\gamma_{0}(t_0,\lambda_*) \ll1$ is sufficiently small; and
	$0 < \overline{p} \le 1$ so that Corollary \ref{welldefined} implies $G_{\mathbf{p}} (t_0)$ is a smooth metric on $M$ for all $\lvert \mathbf{p} \rvert \leq \overline{p} (1-t_0)^{\lvert \lambda_* \rvert}$, and
	Lemma \ref{lemmaforwelldefined} implies $\text{supp	} \phi_{t_0}^* \eta_{\gamma_0} \subset \{f<f_{\mathrm{Bolt}}+ \frac{1}{2}\Gamma_0  \}$. Since $\frac{\Gamma_{0}}{2}<\Gamma_{0}$, this ensures $G_{\mathbf{0}}(t_0)$ satisfies property iii) in Definition \ref{G_0}.
	
		Since $G_{\mathbf{p}}(t_0)$ has bounded curvature, combining \cite[Theorem 1.1]{Chen1} and \cite[Theorem 1.1]{Shi1} to get existence and uniqueness of the Ricci flow, we can let $G_{\mathbf{p}}(t)$ denote the maximal bounded curvature solution to the Ricci flow
	on $M$ with initial data $G_{\mathbf{p}}(t_0)$ at time $t = t_0$.
	Let $T(\mathbf{p} ) = T(\mathbf{p}; \Gamma_0, \gamma_0, t_0)$ denote the minimum of $1$ and the maximal existence time of $G_{\mathbf{p}}(t)$. As we have already stated, we will show that for some $\mathbf{p}\in \mathbb{R}^K$, Ricci flow starting at $G_{\mathbf{p}}(t_0)$ gives us a Ricci flow as in Theorem \ref{finitesing}. This will be done via a Wa{\.z}ewski box argument, and will very closely follow the application of this form of argument found in \cite{Stol1}.

	\subsection{Comments}\label{comments}
	This subsection will comment on how we will use and adapt the work of Stolarski in \cite{Stol1} to prove Theorem \ref{finitesing}. As was mentioned in Subsection \ref{first}, the main difference between the work in \cite{Stol1} and this paper is that we allow for  $\Gamma_{0}$ to get arbitrarily small. The only other difference, which will be of less significance, is that Stolarski considers Ricci flow on compact manifolds. We explain how to adapt the results, and their proofs, found in \cite{Stol1} to account for these two differences.
	
	\subsubsection{Small $\Gamma_{0}$}
	
	Allowing $\Gamma_{0}$ to be small means we have to consider the behaviour of the action of the diffeomorphisms $\phi_t$ near the bolt of $\overline{g}$. As an example of how we account for this difference, we now explain the difference between how Stolarski proves \cite[Lemma 4.5]{Stol1} and how we prove Lemma \ref{lemmaforwelldefined}. Consider the inequality
	\begin{equation}\label{inq}
	\max\left\{0,\phi^*_\tau f-\frac{C}{\phi^*_\tau f}\right\}\leq \partial_\tau(\phi^*_\tau f)\leq \phi_{\tau}^*f,
	\end{equation}
	stated in the proof of Lemma \ref{lemmaforwelldefined}. Heuristically, this inequality says that for $x\in \{f>f_{\mathrm{bolt}}+\Gamma_{0}\}$ for $\Gamma_{0}>\sqrt{C}-f_{\mathrm{bolt}}$, the diffeomorphisms $\phi_{\tau}$ are controlled by the inequality
	 $$\phi^*_\tau f-\frac{C}{\phi^*_\tau f}\leq \partial_\tau(\phi^*_\tau f)\leq \phi_{\tau}^*f.$$
	  However we are interested in small $\Gamma_{0}$. To still take advantage of (\ref{inq}), we pick $\tilde{\tau}_{0}$ large enough so that 
	 $(\phi_{\tilde{\tau}_0}^*f)(x)>\sqrt{C+1}$. We can then say that for $\tau \geq \tilde{\tau}_{0}$, 
	$$\phi^*_\tau f-\frac{C}{\phi^*_\tau f} \leq \partial_\tau(\phi^*_\tau f).$$
	We can then simply follow the proof of Stolarski to get our desired result.

The only other place where the results in \cite{Stol1} need to be altered to take account of $\Gamma_{0}$ being small is in the precise location of the \textit{grafting region} $\Phi_t(\overline{\Omega}_{\eta_{\Gamma_0}})$ defined in \cite[Subsection 5.1]{Stol1}. One of the conclusions of \cite[Lemma 5.5]{Stol1} is (adapting the result in \cite{Stol1} to match our notation) the following estimate of the location of the grafting region.  
\begin{equation}\label{Stol}
\Phi_t ( \overline{ \Omega_{\eta_{\Gamma_0}}} ) \subset \left\{\frac{1}{1 - t} < f < \frac{f_{\mathrm{Bolt}}+\Gamma_0}{1 - t}  \right\},
\end{equation}
When $\Gamma_{0}$ is allowed to be small, the corresponding estimate becomes 
\begin{equation}\label{mine}
\Phi_t ( \overline{ \Omega_{\eta_{\Gamma_0}}} ) \subset \left\{ \tilde{C}\frac{1-t_0}{1 - t} < f < \frac{f_{\mathrm{Bolt}}+\Gamma_0}{1 - t}  \right\},
\end{equation}
	for a suitably small $\tilde{C}>0$. 
	
\subsubsection{Non-compact $M$}

 The only place where this difference requires a different approach is to prove inequality (\ref{mine}). 
To produce the inequality (\ref{Stol}), the analogue of (\ref{mine}) in \cite{Stol1}, Stolarski uses a curvature bound result for the Ricci flow $G_{\mathbf{p}}$ in the form of \cite[Proposition 5.2]{Stol1}. This uses a pseudolocality result for Ricci flow on a compact manifold. Appendix \ref{pseudoapp} presents the corresponding pseudolocality result in the case of Ricci flow on non-compact manifolds, and proves  the corresponding curvature bound result needed to prove inequality~(\ref{mine}).\\

	By altering the statements of the results and their proofs with these differences in mind means we can, with little effort, produce in Section \ref{final} a Ricci flow with singularity modelled on FIK as described in Theorem \ref{finitesing}. By making $\Gamma_{0}$ arbitrarily small, we will produce a Ricci flow with initial metric that is arbitrarily $L^2_{g_{\mathrm{Bolt}}}$ close to $g_{\mathrm{Bolt}}$.

\section{Proof of Theorem 1.1}\label{final}

\begin{proof}[Proof of Theorem \ref{finitesing}]
Let $\Gamma_{0} \in(0, \tilde{\Gamma}_0]$. 
By the Wa{\.z}ewski box argument deployed in the proof of [\cite{Stol1}, Theorem 1.1], there exists a $\mathbf{p}^*$ with $\lvert \mathbf{p}^* \rvert \leq \overline{p}(1-t_0)^{\lvert \lambda_{*}\rvert}$ such that the Ricci flow $G_{\mathbf{p}^*}(t, t_0, \Gamma_{0})$ encounters a local finite time singularity modelled on FIK at $t=1$.

We still have to show that $\lvert G_{\mathbf{p}^*}(t_0, \Gamma_{0})-g_{\mathrm{Bolt}} \rvert_{L^2_{g_{\mathrm{Bolt}}}}$ can be made arbitrarily small if we choose $\Gamma_{0}$ to be arbitrarily small. 
	 By Definition \ref{G_0}, the initial metric $G_{\mathbf{p}^*}(t_0)$ is only a non-trivial perturbation of $g_{\mathrm{Bolt}}$ on the set $\{f\le \tilde{C}\}$, where $\tilde{C} \rightarrow 0$ as $\Gamma_{0}\rightarrow 0$. We now show that if by choosing $\Gamma_{0}$ and $\gamma_{0}$ be to sufficiently small, we can make the quantity
$$\lvert G_{\mathbf{p}^*}(t_0, \Gamma_{0})-g_{\mathrm{Bolt}} \rvert_{L^2_{g_{\mathrm{Bolt}}}}$$
as small as we desire. By Definition \ref{G_0}, there is a constant $B>0$ independent of $\Gamma_{0}\leq \tilde{\Gamma}_0$  such that for all $0< 1-t_0 \ll 1$ sufficiently small depending on $\Gamma_{0}$ we have
$$G_{\mathbf{0}}(t_0, \Gamma_{0}) \leq Bg_{\mathrm{Bolt}},$$
on the set $\{f\le f_{\mathrm{Bolt}}+\tilde{C}\}$.

By Lemma \ref{G_p-G-0}, if $0< 1-t_0 \ll 1$ is chosen small enough and $\gamma_{0}(t_0)$ small enough, there exists $C=C(\lambda_{*})>0$ such that $$\lvert G_{\mathbf p ^*} (t_0, \Gamma_{0}) - G_{\mathbf 0} (t_0, \Gamma_{0}) \rvert_{G_{\mathbf 0} (t_0, \Gamma_{0})}\leq C \overline{p}\gamma_{0}^{\lvert \lambda_{*} \rvert} .$$ 
 Thus,  
\begin{align*}
\lvert G_{\mathbf{p}^*}(t_0, \Gamma_{0}) - g_{\mathrm{Bolt}} \rvert_{g_{\mathrm{Bolt}}} &\leq \lvert G_{\mathbf{p}^*}(t_0, \Gamma_{0})-G_{\mathbf{0}}(t_0, \Gamma_{0})\rvert_{g_{\mathrm{Bolt}}} +\lvert G_{\mathbf{0}}(t_0, \Gamma_{0}) \rvert_{g_{\mathrm{Bolt}}}+ \lvert g_{\mathrm{Bolt}} \rvert_{g_{\mathrm{Bolt}}}\\
&\leq \frac{1}{B^2}(\lvert G_{\mathbf{p}^*}(t_0, \Gamma_{0})-G_{\mathbf{0}}(t_0, \Gamma_{0})\rvert_{G_{\mathbf{0}}(t_0, \Gamma_{0})} +\lvert G_{\mathbf{0}}(t_0, \Gamma_{0}) \rvert_{G_{\mathbf{0}}(t_0, \Gamma_{0})})+ 2\\
&\leq  \frac{1}{B^2}(C\overline{p}\gamma_{0}^{\lvert \lambda_{*} \rvert}+2)+2 \coloneqq D	
\end{align*}

As a result,
$$\lvert G_{\mathbf{p}^*}(t_0, \Gamma_{0})-g_{\mathrm{Bolt}} \rvert_{L^2_{g_{\mathrm{Bolt}}}}\leq D\sqrt{\text{Vol}_{g_{\mathrm{Bolt}}}(\{f\le f_{\mathrm{Bolt}}+\tilde{C}\})}.$$ 
Letting $\Gamma_{0} \rightarrow 0$, which implies $\tilde{C}\rightarrow 0$, gives 
$$\lvert G_{\mathbf{p}^*}(t_0, \Gamma_{0})-g_{\mathrm{Bolt}} \rvert_{L^2_{g_{\mathrm{Bolt}}}} \rightarrow 0.$$
\end{proof}

\appendix
 \section{Construction of $G_{\mathbf{0}}(t_{0})$}\label{construct}
Recall that the FIK soliton is denoted by $(M,\overline{g},f)$, and that $C_{\mathrm{FIK}}$ is the cone over the sphere $\left(S^3, g_{S^3_\mathrm{FIK}} \coloneqq \frac{1}{\sqrt{2}}(\sigma_{1}^2+\sigma_{2}^2)+\frac{1}{2}\sigma_{3}^2\right)$.  For $R>0$, define $C_{R}(S^3)\coloneqq (R, \infty) \times S^3$ and equip it with the restriction of the metric $g_{C_{\mathrm{FIK}}}$ to $C_{R}(S^3)$. The construction of $G_{\mathbf{0}}$ is inspired by the construction of \cite[Appendix A]{Stol1}.\\

\begin{lemma}\label{conetoTB} Recall the diffeomorphism $\Psi$ from Theorem \ref{FIKthm}. Since $\Psi^* g_{\mathrm{Bolt}}$ is a cohomogeneity one $U(2)$-symmetric metric, it can be written in the form of equation (\ref{cohomform}):
	$$\Psi^*g_{\mathrm{Bolt}}=ds^2+ b^2(s)\left(\sigma_{1}^2+\sigma_{2}^2\right)+c^2(s)\sigma_{3}^2, \qquad s>0.$$
	There exists an $\tilde{R}>0$ such that for all $0<R\leq \tilde{R}$ there exists smooth functions $b_{R},c_{R}:~(0,\infty)\rightarrow (0,\infty)$ which satisfy
	$$b_R(s) \left\{ \begin{array}{ll}
		= \frac{s}{\sqrt[4]{2}}, 		& \text{ if } 0 < s \le R,	\\
		=  b(s), 	& \text{ if } 3R \le s,	\\
	\end{array} \right.$$
	$$c_R(s) \left\{ \begin{array}{ll}
		= \frac{s}{\sqrt{2}}, 		& \text{ if } 0 < s \le R,	\\
		= c(s), 	& \text{ if } 3R \le s,	\\
	\end{array} \right.$$
	and so that the curvature of the metric 
	$$g_R=ds^2+b_{R}^2(s)(\sigma_{1}^2+\sigma_{2}^2)+c_{R}^2(s)\sigma_{3}^2,$$ and on $\{s\geq \frac{R}{2}\}$ satisfies
	$$ \lvert \nabla_{g_R}^{m}\mathrm{Rm}[g_R]\rvert_{g_R}\leq C=C(m,R).$$
	Furthermore, $g_R\leq \Psi^* g_{\mathrm{Bolt}}$.
\end{lemma}
\begin{proof}
	Since $b(0)>0$ and $\partial_s c(0)=1>\frac{1}{\sqrt{2}}$, there exists an $\tilde{R}>0$ such that for all $0<R\leq \tilde{R}$ there exists smooth functions $b_{R},c_{R}:(0,\infty)\rightarrow (0,\infty)$ with $b_R\leq b$ and $c_R\leq c$ and satisfying the properties written in the statement of the theorem. This implies $g_R\leq \Psi^* g_{\mathrm{Bolt}}$. The curvature bounds follow as $\lvert \nabla^m_{\Psi^*g_{\mathrm{Bolt}}}\text{Rm}_{\Psi^*g_{\mathrm{Bolt}}} \rvert_{\Psi^*g_{\mathrm{Bolt}}}$ is bounded for all $m$.
	\end{proof}

\begin{lemma}\label{B}
	There exists a $B=B(R)>0$ such that $(1-t)\phi_{t}^*\overline{g} \leq Bg_{\mathrm{Bolt}}$ on $\{s\leq R\}$ for all $t\in[0,1)$. 
		\end{lemma}

\begin{proof}
The metric $\overline{g}$ evolves under the Ricci flow as
\begin{equation}\label{FIKev}
	(1-t)\phi_{t}^*\overline{g}=(1-t)[d(s\circ \phi_{t})^2+ b_{\mathrm{FIK}}(s\circ \phi_{t})^2(\sigma_{1}^2+\sigma_{2}^2)+c_{\mathrm{FIK}}(s\circ \phi_{t})^2\sigma_{3}^2].
\end{equation}
for $s>0$.
By Theorem \ref{FIKthm} and \cite[Theorem 1.1]{Cao}, there are constants $A, B, \tilde{B}, C, D, \tilde{D}, E, F>0$ such that
$$As^2\leq f(s)\leq Bs^2+\tilde{B}, Cs\leq b_{\mathrm{FIK}}(s) \leq Ds+\tilde{D}, Es \leq c_{\mathrm{FIK}}(s)\leq Fs.$$
Rewriting one of the conclusions of \cite[Lemma 2.33]{Stol1} using the variable $\tau(t)= -\text{ln}(1-t)$,  it follows that 
$$\partial_\tau(\phi^*_\tau f)\leq \phi_{\tau}^*f,$$ for all $\tau \in \mathbb{R}.$
Since $\phi_0=\text{Id}$, it follows that $f(s\circ \phi_{t}) \leq \frac{f(s)}{1-t}$, and so there exists $G>0$ such that
$$s\circ \phi_{t} \leq G\frac{s}{\sqrt{1-t}}.$$
Since $b_{\mathrm{FIK}}$ is non-decreasing, there exist $H, I, \tilde{I}>0$ such that
$$\sqrt{1-t}b_{\mathrm{FIK}}(s\circ \phi_{t}) \leq b_{\mathrm{FIK}}(H\frac{s}{\sqrt{1-t}}) \leq Is+\tilde{I}.$$
Similarly, $\sqrt{1-t}c_{\mathrm{FIK}}(s\circ \phi_{t}) \leq Ks$ for some $K>0$.
So the coefficients of (\ref{FIKev}) are bounded on $\{s\leq R\}$. Therefore $(1-t)\phi_{t}^*\overline{g}$ is dominated by $g_{\mathrm{Bolt}}$ on $\{s\leq R\}$.
\end{proof}

The next lemma constructs an interpolation between  $\overline{g}$ and the metric of Lemma \ref{conetoTB}. Since Theorem \ref{FIKthm} gives that for any $R_0>0$, $(1-t_0)\Psi^* \phi_{t_0}^* \overline{g}-g_{C_{FIK}} \rightarrow 0$ in $C_{loc}^{\infty}(\overline{C_{R_0}(S^3)}, g_{C_{FIK}})$ as $t \nearrow 1$, following the proof of \cite[Lemma A.4]{Stol1} and using that $g_R\leq \Psi^*g_{\mathrm{Bolt}}$ in Lemma \ref{conetoTB}, and Lemma \ref{B} yields the following lemma. 
\begin{lemma} \label{interp}
	Let $\tilde{R}>0$ be as in Lemma \ref{conetoTB}. 
	Consider $0<R_3\leq \tilde{R}$ and let $\Psi_t = \phi_t \circ \Psi$ be as in Theorem \ref{FIKthm}.
	Let $R_0, R_1, R_2 \in \mathbb{R}$ be such that
	$$0 < R_0 < R_1 < R_2 < R_3.$$
	Let $\eta(r) : (0, \infty) \to [0,1]$ be a smooth bump function that decreases from 1 to 0 over the interval $(R_1, R_2)$.
	For all $0 \le t_0 < 1$, consider the metric $\check{G}(t_0)$ on $C_{R_0}( S^3)$ given by 
	$$\check G(t_0)  = \eta ( 1 - t_0) \Psi^* \phi_{t_0}^* \overline{g} + ( 1 - \eta) g_{R_3}$$
	where $g_{R_3}$ is the metric from Lemma \ref{conetoTB}. Then
	$$\check G(t_0) \xrightarrow[t_0 \nearrow 1]{} g_{C_{\mathrm{FIK}}}
	\quad \text{in } C^\infty \left( C_{R_0}(S^3) \setminus  \overline{C_{R_3}( S^3 )} , g_{C_{g_{\mathrm{FIK}}}} \right)$$
	and, for all $m \in \mathbb{N}$,
	if $0 < 1 - t_0 \ll 1$ is sufficiently small depending on $R_0, R_1, R_2, R_3$, 
	then
	$\check G(t_0)$ satisfies the curvature estimate
	$$\lvert \check \nabla^m \check {\mathrm{Rm}} \rvert_{\check G(t_0) } \leq C=C(m, R_0, R_1, R_2, R_3)
	\qquad \text{on } C_{R_0}( S^3) .$$ 
	Furthermore, there is $B>0$ independent of $R_3\leq \tilde{R}$ such that $\check{G}(t_0) \leq B\Psi^*g_{\mathrm{Bolt}}$ on $C_{R_0}( S^3)$.
\end{lemma}
Let $\check G(t_0)$ and $\Psi$ be as in Lemma \ref{interp}. The metric $(\Psi^{-1} )^* \check G(t_0)$ on $\Psi( C_{R_0} ( S^3)  ) \subset M$ can be extended by $( 1 - t_0) \phi_{t_0}^* \overline{g}$ on the complement of $\Psi ( C_{R_0} ( S^3)  )$ to yield a metric on $M$.
We denote this metric by $G(t_0)$. Lemma \ref{B} yields the  dominance of $G(t_0)$ by $g_{\mathrm{Bolt}}$ on $\{f\leq f(\Psi(R_0)\}$. Since $\Psi$  maps $\{r=\mathrm{constant}\}\subset C(S^3)$ to $\{f=\mathrm{constant}\} \subset(\mathbb{CP}^2\setminus \{\mathrm{pt}\})\backslash \{f=0\}$, and preserves the order, it is clear that the following proposition holds.

\begin{proposition} \label{constructiondone}
	Let $0<\Gamma_{0}\leq f(\Psi(\tilde{R_0}))=\tilde{\Gamma}_0$, let $0 < 1 - t_0 \ll 1$ be sufficiently small depending on $\Gamma_{0}$.
	If $f(\Psi(R_0))=f_{\mathrm{Bolt}}+\frac{\Gamma_{0}}{2}$, $f(\Psi(R_1))=f_{\mathrm{Bolt}}+\Gamma_{0}$, $f(\Psi(R_2))=f_{\mathrm{Bolt}}+16\Gamma_{0}$, $f(\Psi(R_3))=f_{\mathrm{Bolt}}+32\Gamma_{0}$ then there exists $C=C(\Gamma_{0},m)$ such that $G(t_0)$ has the properties listed in Definition \ref{G_0} of $G_{\mathbf{0}}(t_0, \Gamma_{0})$.
\end{proposition}

\section{Pseudolocality results}\label{pseudoapp}
This subsection will prove the analogue of the curvature bound result \cite[Proposition 5.2]{Stol1} in the non-compact setting. Thus, we show that if $0\leq 1-t_0 \ll 1$ is sufficiently small then, for any $c>0$, the curvature $\lvert \mathrm{Rm}[G_{\mathbf{p}}(t)]\rvert_{G_{\mathbf{p}}(t)}$ is bounded on $\{f\geq f_{\mathrm{Bolt}}+\frac{1}{2}\Gamma_{0}+c\} \times [t_0, T(\mathbf{p})]$. This shows that any finite time singularity of $G_{\mathbf{p}}(t)$ would have to be local (i.e. the curvature does not blow up on the entirety of $M$). Proving the curvature bound will involve the use of a pseudolocality result of Peng \cite{Peng}. We denote by $\omega$ the volume of the sphere $S^4$ with unit radius.
\begin{theorem}[\cite{Peng}]  \label{pseudolocalitynoncompact}
	There exist $\epsilon, \delta > 0$  with the following property:
	If $(M^4, g(t) )$ is a smooth complete Ricci flow on $M$ with bounded curvature on each time slice and defined for $t \in [t_0, T)$ that, at initial time $t_0$, satisfies
	$$\lvert \mathrm{Rm}\rvert_{g(t_0)}(x) \le \frac{1}{ r_0^2} \qquad \text{ for all  } x \in B_{g(t_0)} ( x_0, r_0)$$
	and
	$$\mathrm{Vol}_{g(t_0) } B_{g(t_0)} (x_0, r_0) \ge ( 1 - \delta ) \omega r_0^4 \qquad \text{ for all } x \in B_{g(t_0)} ( x_0, r_0)$$
	for some $x_0 \in M$ and $r_0 > 0$,
	then
	$$\lvert \mathrm{Rm}\rvert_{g(t)}(x,t) \leq \frac{1}{\epsilon^2 r_0^2} \qquad \text{for all } x \in B_{g(t)} (x_0, \epsilon r_0), t \in [t_0, \min \{ T , t_0 +  (\epsilon r_0)^2 \} ).$$
\end{theorem}

\begin{proposition}\label{pseudo5}
	Let $c, \Gamma_{0}>0$. There exists $\mathcal{K}_0=\mathcal{K}_0(\Gamma_{0})$ such that if $0<1-t_0 \ll 1$ is small enough depending on $\Gamma_{0}$, then for all $\lvert \mathbf{p} \rvert \leq \overline{p}e^{\lambda_{*} \tau_{0}}$,
	$$\lvert \mathrm{Rm}[G_{\mathbf{p}}(t)]\rvert_{G_{\mathbf{p}}(t)} (x,t) \leq \mathcal{K}_0,$$
	for all $(x,t)\in \{f\geq f_{\mathrm{Bolt}}+ \frac{1}{2}\Gamma_{0}+c\} \times [t_0, T(\mathbf{p})].$
\end{proposition}

\begin{proof}
	Let $\epsilon, \delta  > 0$ be as in Theorem \ref{pseudolocalitynoncompact}. Since $\text{supp	} \phi_{t_0}^* \eta_{\gamma_0} \subset \{f<f_{\mathrm{Bolt}}+ \frac{1}{2}\Gamma_0  \}$, throughout $\{ f > f_{\mathrm{Bolt}}+\Gamma_0 / 2\}$, we have
	$G_{\mathbf{p} } (t_0) = G_{\mathbf{0}} (t_0)$ for all $\mathbf{p}$. Therefore, for some $C = C(\Gamma_0)$, 
	$$\lvert \mathrm{Rm}[ G_{\mathbf{p} }(t_0) ] \rvert_{G_{\mathbf{p} }(t_0) } \le C$$ on $\{x \in M : f(x) \ge f_{\mathrm{Bolt}}+\Gamma_0 /2 \} $, 
	by Definition \ref{G_0} provided $0 < 1 - t_0 \ll 1$ is sufficiently small. Using that $G_{\mathbf{0}}(t_0)=g_{\mathrm{Bolt}}$ on $\{f>\tilde{C}\}$ for some $\tilde{C}>0$, it follows that $ G_{\mathbf{0}}(t_0)$ is asymptotic to $ds^2+4s^2(\sigma_{1}^2+\sigma_{2}^2)+\alpha \sigma_{3}^2$ for some $\alpha>0$. Hence, we can find an $0 < r_0 \ll 1$ sufficiently small such that  $\text{Vol} \,B_{ g_{\mathrm{Bolt}}} (x, r_0) \ge ( 1 - \delta) \omega r_0^4
	\text{ for all } x \in M$. 
	Furthermore, the restriction of the metrics $G_{\mathbf{p} } (t_0)$ to $\{ f > f_{\mathrm{Bolt}}+\Gamma_0 / 2\}$ stay within an arbitrarily small $C^0$-neighbourhood for all $0 < 1 - t_0 \ll 1$ sufficiently small by property ii) and iii) of Definition \ref{G_0}.
	Therefore, restricting to the set $\{f\geq f_{\mathrm{Bolt}}+ \frac{1}{2}\Gamma_{0}+c\}$, there exists $0 < r_0 \ll 1$ sufficiently small 
	depending on $\Gamma_0$
	so that, for all $0 < 1 - t_0 \ll 1$ sufficiently small,
	\begin{gather*}
		C \le \frac{1}{r_0^2} ,\\
		B_{G_{\mathbf{p} } (t_0)}(x, r_0) \subset \{ f >f_{\mathrm{Bolt}}+ \Gamma_0/2 \}
		\qquad \text{for all } x \in \{f\geq f_{\mathrm{Bolt}}+ \frac{1}{2}\Gamma_{0}+c\}, \text{ and}	\\
		\text{Vol} \, B_{G_{\mathbf{p}} (t_0)} (x, r_0) \ge ( 1 - \delta) \omega r_0^4
		\qquad \text{for all } x \in \{f\geq f_{\mathrm{Bolt}}+ \frac{1}{2}\Gamma_{0}+c\}.
	\end{gather*}
	If $0 < 1 - t_0 \ll 1$ is also small enough so that
	$$T( \mathbf{p}) - t_0 \le 1 - t_0 \le \epsilon^2 r_0^2,$$ then 
	$$\lvert \mathrm{Rm}[ G_{\mathbf{p} }(t) ] \rvert_{G_{\mathbf{p} }(t) }(x)  \le \frac{1}{ \epsilon^2 r_0^2 } \qquad \text{for all } (x,t) \in \{f\geq f_{\mathrm{Bolt}}+ \frac{1}{2}\Gamma_{0}+c\} \times [t_0, T(\mathbf{p})).$$
	Taking $\mathcal{K}_0 = ( \epsilon r_0)^{-2}$ completes the proof.
\end{proof}

\bibliography{refs}

\providecommand{\bysame}{\leavevmode\hbox to3em{\hrulefill}\thinspace}
\providecommand{\MR}{\relax\ifhmode\unskip\space\fi MR }
\providecommand{\MRhref}[2]{%
  \href{http://www.ams.org/mathscinet-getitem?mr=#1}{#2}
}
\providecommand{\href}[2]{#2}
\begin{thebibliography}{10}

\bibitem{App1}
Alexander Appleton, \emph{Eguchi-{H}anson singularities in {$U(2)$}-invariant
  {R}icci flow}, Peking Math. J. \textbf{6} (2023), no.~1, 1--141.

\bibitem{Cao}
Huai-Dong Cao and Detang Zhou, \emph{On complete gradient shrinking {R}icci
  solitons}, J. Differential Geom. \textbf{85} (2010), no.~2, 175--185.
  \MR{2732975}

\bibitem{Chen1}
Bing-Long Chen and Xi-Ping Zhu, \emph{Uniqueness of the {R}icci flow on
  complete noncompact manifolds}, J. Differential Geom. \textbf{74} (2006),
  no.~1, 119--154.

\bibitem{Fra2}
Francesco Di~Giovanni, \emph{Ricci flow of warped {B}erger metrics on
  {$\Bbb{R}^4$}}, Calc. Var. Partial Differential Equations \textbf{59} (2020),
  no.~5, Paper No. 162, 37.

\bibitem{FIK}
Mikhail Feldman, Tom Ilmanen, and Dan Knopf, \emph{Rotationally symmetric
  shrinking and expanding gradient {K}\"{a}hler-{R}icci solitons}, J.
  Differential Geom. \textbf{65} (2003), no.~2, 169--209.

\bibitem{Cla1}
Gustav Holzegel, Thomas Schmelzer, and Claude Warnick, \emph{Ricci flows
  connecting {T}aub-bolt and {T}aub-{NUT} metrics}, Classical Quantum Gravity
  \textbf{24} (2007), no.~24, 6201--6217.

\bibitem{Dan2}
James Isenberg, Dan Knopf, and Nata\v{s}a \v{S}e\v{s}um, \emph{Non-{K}\"{a}hler
  {R}icci flow singularities modeled on {K}\"{a}hler-{R}icci solitons}, Pure
  Appl. Math. Q. \textbf{15} (2019), no.~2, 749--784.

\bibitem{Oxford1}
M~Johar, \emph{Ricci flow in {M}ilnor frames}, Ph.D. thesis, University of
  Oxford, 2019.

\bibitem{Peng}
Peng Lu, \emph{A local curvature bound in {R}icci flow}, Geom. Topol.
  \textbf{14} (2010), no.~2, 1095--1110.

\bibitem{Bolt}
Don~N. Page, \emph{Taub-{NUT} instanton with an horizon}, Physics Letters B
  \textbf{78} (1978), no.~2, 249--251.

\bibitem{Shi1}
Wan-Xiong Shi, \emph{Deforming the metric on complete {R}iemannian manifolds},
  J. Differential Geom. \textbf{30} (1989), no.~1, 223--301.

\bibitem{Stol1}
Maxwell Stolarski, \emph{Closed {R}icci flows with {S}ingularities modeled on
  {A}symptotically {C}onical {S}hrinkers}, 2022, preprint,
  \url{https://arxiv.org/abs/2202.03386}.

\bibitem{Wazewski}
Tadeusz Wa\.zewski, \emph{Sur un principe topologique de l'examen de l'allure
  asymptotique des int\'egrales des \'equations diff\'erentielles ordinaires},
  Ann. Soc. Polon. Math. \textbf{20} (1947), 279--313.

\end{thebibliography}
\bibliographystyle{amsplain}

\end{document}